\documentclass[10pt,draft,reqno]{amsart}
     \makeatletter
     \def\section{\@startsection{section}{1}%
     \z@{.7\linespacing\@plus\linespacing}{.5\linespacing}%
     {\bfseries%\normalfont\scshape
     \centering
     }}
     \def\@secnumfont{\bfseries}
     \makeatother
\newtheorem{theorem}{Theorem}[section]
\newtheorem{lemma}[theorem]{Lemma}

\theoremstyle{definition}

\theoremstyle{remark}
\newtheorem{remark}[theorem]{Remark}
\numberwithin{equation}{section}
\newcommand{\cG}{{\mathcal G}}

\newcommand{\cK}{{\mathcal K}}

\newcommand{\cM}{{\mathcal M}}
\newcommand{\cN}{{\mathcal N}}

\newcommand{\cP}{{\mathcal P}}

\newcommand{\cS}{{\mathcal S}}

\newcommand{\bbE}{{\mathbb E}}

\newcommand{\bbP}{{\mathbb P}}
\newcommand{\bbR}{{\mathbb R}}

\begin{document}

\title[Random matrices]{Large Deviations for Random Matrices}
\author{Sourav Chatterjee}
\address{Courant Institute, New York University, New York, NY, 10012, USA}
\email{sourav@cims.nyu.edu}
\urladdr{http://www.cims.nyu.edu/~sourav/}
\author{S.R.S.Varadhan}
\thanks{\noindent Sourav Chatterjee's  research was partially supported by NSF grants DMS-0707054 and DMS-1005312, and a Sloan Research Fellowship.}
\thanks{S. R. S. Varadhan's research was partially supported by NSF grants DMS-0904701 and OISE-0730136.}
\address{Courant Institute, New York University, New York, NY, 10012, USA}
\email{varadhan@cims.nyu.edu}
\urladdr{http://www.cims.nyu.edu/~varadhan/}


\begin{abstract}
We prove a large deviation result for  a random symmetric $n\times  n$ matrix with independent identically distributed entries 
to have a few eigenvalues of size $n$. If the spectrum $\mathcal S$ survives when the matrix is rescaled by a factor of $n$,
it can only be the eigenvalues of a Hilbert-Schmidt kernel $k(x,y)$ on $[0,1]\times [0,1]$. The rate function for $k$ is 
$I(k)=\frac{1}{2}\int h(k(x,y) dxdy$  where $h$ is the Cram\'{e}r rate function for the common distribution of the entries that is assumed to have a tail decaying faster than any Gaussian.
The large deviation for $\mathcal S$ is then obtained by contraction.
\end{abstract}

\maketitle

\noindent 

\section{ Introduction}
We are interested in the large deviation behavior of random $n\times n$ symmetric matrices $X(\omega)=\{x_{i,j}(\omega)\}$ where the entries for $j\ge i$ are independent identically distributed real random variables having  a common distribution $\mu$.   Let $\Lambda_n(\omega)=\{\lambda^n_j(\omega);1\le j\le n\}$ be the set of  $n$ real eigenvalues of this random matrix.  If we assume that  the  mean is  $0$ and the variance is equal to $\sigma^2$, according to a theorem of Wigner \cite{wigner58}, when divided by $\sqrt{n}$, the empirical distribution of these eigenvalues converges in probability  to the semi-circle law, i.e.\ for any bounded continuous function $f: \bbR \to \bbR$ 

 $$
 \lim_{n\to\infty}\frac{1}{n}\sum_{j=1}^n f\biggl(\frac{\lambda_j(\omega)}{\sqrt{n}}\biggr)=\int f(y)\phi_\sigma(y) dy
 $$
 in probability, where
 \begin{align*}
 \phi_\sigma(y)=
 \begin{cases}
 0\ &\ {\rm if}\ |x|\ge 2\sigma,\\
 \frac{1}{2\pi\sigma^2} \sqrt{4\sigma^2-x^2}\ &\ {\rm if}\ |x|\le 2\sigma.
 \end{cases}
 \end{align*}
 In fact it is known \cite{bai99} that under suitable additional assumptions, for any $\epsilon > 0$,  
 $$
 \bbP\bigl[\max_{1\le j\le n}|\lambda^n_j|\ge (2\sigma+\epsilon)\sqrt{n}\bigr]\to 0.
 $$
 Large deviation results for the distribution of eigenvalues to be different from the Wigner distribution has been considered by Ben Arous and Guionnet \cite{bg97}. If we decide to divide by $n$ rather  than $\sqrt{n}$, and denote the resulting  spectrum by $S_n(\omega)=\frac{1}{n}\Lambda_n(\omega)$, then
 $$
 \bbP[\sup_{\lambda\in S_n} |\lambda|\ge \epsilon]\to 0
 $$
 If we drop the assumption that the mean is $0$, then there will be one large eigenvalue $\lambda_{\it max} $  of size $n$ in $\Lambda_n$, with $\frac{\lambda_{\it max}}{n}\simeq \int x d\mu$  and the remaining eigenvalues will follow the semi-circle law as before when divided  by
 $\sqrt n$.  It is also possible  to consider scaling by $n^\alpha$ for $\frac{1}{2}<\alpha< 1$, but they would likely behave like moderate deviations and exhibit Gaussian behavior with rate $n^{2\alpha}$.

Let $\cS$ denote the set of all finite  or countable collections $S=\{\lambda_j\}$  of positive as well as negative real  numbers, with repetitions allowed,  that  have the property
 $$
 \sum_{\lambda\in S}  |\lambda|^2= \sum_j |\lambda_j|^2<\infty. 
 $$ 
They are  all  possible spectra of self adjoint  Hilbert-Schmidt operators. The topology on  $\cS$ will be  ordinary convergence as real numbers of the corresponding eigenvalues  outside any arbitrarily small interval around $0$. Equivalently it is the minimal topology such that for any bounded continuous  function $f$ that is $0$ in some interval around $0$, the sum $\sum_{\lambda\in \cS} f(\lambda)$   is continuous as a map of $\cS\to \bbR$.  One can easily  construct a metric for this topology.  We enumerate separately the positive and negative  values in $S$ in decreasing order of their absolute values  as $\{u_j^+\},\{u_j^-\}$. If  one or both of them is only a finite set or empty we augment it by adding  $0$'s. If two points $S,S'$  in $\cS$ are  enumerated as $\{u_j^+\},\{u_j^-\}$ and $\{v_j^+\},\{v_j^-\}$, we define the distance  
$$
d(S,S'):=\sum_{j=1}^\infty \frac{|u_j^+-v_j^+|}{2^{j+1}}+\sum_{j=1}^\infty\frac{|u_j^--v_j^-|}{2^{j+1}}
$$
The random matrix yields a random spectrum and after normalization by $n$, we obtain a random element of $\cS$. Its distribution yields a  sequence  $\{P_n\}$ of probability measures on $\cS$. We will prove a large deviation result for $P_n$ with rate $n^2$.  We assume that for all $\theta > 0$, 
\begin{equation}\label{eqbound}
\int  \exp[\theta\,x^2] \mu(dx)<\infty. 
\end{equation}
In particular  this implies that the moment generating function $M(\theta)=\int e^{\theta\,x}\mu(dx)$ is finite  for all $\theta$ and satisfies
 $$
 \limsup_{|\theta|\to\infty}\frac{1}{\theta^2}\log M(\theta)=0
 $$
and the conjugate rate function of Cram\'{e}r
\begin{equation}\label{cramer}
 h(x)=\sup_\theta[\theta\, x-\log M(\theta)]
\end{equation}
 satisfies
 $$
  \liminf_{|x|\to\infty}\frac{h(x)}{x^2}=+\infty.
 $$
 The  condition (\ref{eqbound}) is important, because an eigenvalue of size $n$ can be produced by a single entry of size $n$ in the random matrix  and we would like this to have probability that is super-exponentially small in the scale $n^2$.  
 
\medskip\noindent
The random  symmetric $n\times n$  matrix $X(\omega)=\{x_{i,j}(\omega)\}$  is first mapped into a symmetric  kernel $k(x,y,\omega)$  
\begin{equation}\label{def:k}
 k(x,y,\omega)=\sum_{i,j=1}^n x_{i,j}(\omega){\bf 1}_{J^n_i}(x){\bf 1}_{J^n_j}(y)
\end{equation}
where $J^n_i$ is the interval $[\frac{i-1}{n},\frac{i}{n}]$. This induces a family of probability measures $Q_n$ on  the space of symmetric kernels $k(x,y)$ on $D=[0,1]\times [0,1]$. We will restrict ourselves to $\cK=\{k:\int\int_D |k(x,y)|^2 \,dxdy <\infty\}$. Then $k(x,y)$ defines on $L^2[0,1]$ a Hilbert-Schmidt operator which has a countable spectrum with $0$ as  the only  limit point. 
Actually for each fixed $n$ the range of the map $X\to k$ is a finite dimensional subspace of  simple functions $\cK_n\subset \cK$. The nonzero spectrum of $k(x,y,\omega)$ is the same as that of $\{\frac{x_{i,j}}{n}\}$ and we can obtain  $P_n$  from $Q_n$ through the natural map $\cK\to \cS$ that takes any $k$ to its  set of  eigenvalues.  We define on $\cK$ the following rate function
\begin{equation}\label{rate}
 I(k(\cdot,\cdot))=\frac{1}{2}\int_0^1\int_0^1 h(k(x,y))dxdy
\end{equation}
with $h$ given by (\ref{cramer}).

\medskip
Any  permutation $\sigma \in \Pi(n)$ of the rows and columns of $X$, mapping $\{x_{i,j}\}\to \{x_{\sigma(i),\sigma(j)}\}$ leaves the set of eigen-values of $X$ invariant and the group $G$ of measure preserving transformations $\sigma$ of $[0,1]$ onto itself   lifts to an action on $\cK$  mapping  $k\to \sigma k$ where  $ \sigma k (x,y):=k(\sigma x,\sigma y)$. The map $k\to\sigma k$ leaves $I(k(\cdot,\cdot))$ as well as the spectrum $S(k)$ of $k$ invariant. For establishing the large deviations of $P_n$ on $\cS$ it is therefore enough to prove a large deviation principle for the images ${\tilde Q}_n$ of $Q_n$  on $\tilde{\cK}=\cK/G$.
 
 \medskip
 We will be working with the space $\cK$ of symmetric kernels $k(x,y)$ on $D$.   If $I(k)<\infty$, then the operator defined by $k(x,y)$ on $L^2[0,1]$ is Hilbert-Schmidt and has a countable spectrum with $0$ as  the only  limit point. 
 We need to show some sort of continuity of the map ${\tilde\cK}\to \cS$  mapping the $G$-orbit $\tilde k$ of   $k$ to its spectrum $S(k)$,  in order to transfer the large deviation result from $\tilde \cK$ to $\cS$. This requires a topology on $\tilde{\cK}$ that will be inherited from $\cK$.  The weak topology on $\cK$ turns out to be too weak and the strong or $L^1$ topology too strong. What works is the  topology induced by the cut  metric
\begin{equation}\label{cut1}
d_\square(k_1,k_2)=\sup_{|\phi|\le 1\atop|\psi |\le 1}\biggl|\int (k_1(x,y)-k_2(x,y))\phi(x)\psi(y) dxdy\biggr| \ ,
\end{equation}
$\phi,\psi$ being Borel measurable functions on $[0,1]$.  Equivalently
\begin{equation}\label{cut2}
d_\square(k_1,k_2)=\sup_{A\times B}\biggl|\int_{A\times B} (k_1(x,y)-k_2(x,y))dxdy\biggr|
\end{equation}
where the supremum is taken over all Borel subsets $A, B$ of $[0,1]$. The induced metric on $\tilde{\cK}$ is
$$
d_\square(k_1,k_2)=\inf_\sigma d_\square(\sigma k_1,k_2)=\inf_\sigma d_\square( k_1, \sigma k_2)=\inf_{\sigma_1,\sigma_2} d_\square( \sigma_1k_1, \sigma_2 k_2)
$$
where $\sigma k(x,y)=k(\sigma x,\sigma y)$.  We can define on $\cS$ the rate function
 $$
 J(S)=\inf_{k: S(k)=S} I(k).
 $$
 Our main result is the following. 
 \begin{theorem}\label{main}
 Under assumption (\ref{eqbound})  the sequence of measures $P_n$ on $\cS$ satisfies a large deviation property with rate function $J(S)$, i.e for closed $C\subset \cS$
 $$
 \limsup_{n\to\infty}\frac{1}{n^2}\log P_n (C)\le -\inf_{S\in C} J(S)
 $$
 and for $U\subset\cS$ that are open
 $$
 \liminf_{n\to\infty}\frac{1}{n^2}\log P_n (U)\ge -\inf_{S\in U} J(S).
 $$
 \end{theorem}
 This is based on a large deviation principle for $\tilde Q_n$  on $\tilde\cK$ in the cut topology with rate function $I(k)$.
  \begin{theorem}\label{main2}
 Under assumption (\ref{eqbound}) the sequence of measures $\tilde Q_n$ on $\tilde \cK$ satisfies a large deviation property with rate function $I(\tilde k)$, i.e for closed $C\subset \tilde\cK$
 $$
 \limsup_{n\to\infty}\frac{1}{n^2}\log \tilde Q_n (C)\le -\inf_{\tilde k\in C} I(\tilde k)
 $$
 and for $U\subset\tilde \cK$ that are open
 $$
 \liminf_{n\to\infty}\frac{1}{n^2}\log \tilde Q_n (U)\ge -\inf_{\tilde k\in U} I(\tilde k).
 $$
 \end{theorem}
To make the connection we need  the map $k\to S(k)$ to be  continuous in the cut topology. It is valid, provided we restrict it to
sets $\cK^\ell$ of the form $\cK^\ell=\{k:|k(x,y)|\le \ell\}$ for some $\ell<\infty$. This requires truncation  at level $\ell$ and then removing the cut-off. Condition  (\ref{eqbound}) provides super-exponential bounds in the Hilbert-Schmidt norm for the error due to cutoff and that is used to complete the proof.
 
 \medskip
Similar methods were used in \cite{cv10} to study the large deviation behavior of  the number of triangles or other finite subgraphs in random graphs as the number of vertices goes to $\infty$ but the probability of an edge being connected remains fixed at some $p>0$. This will correspond to each $x_{i,j}$ taking the values $0$ or $1$.

 \section{Some  useful lemmas.}
 
 \medskip
 We will be working with partitions of the unit interval $[0,1]$ into  a finite disjoint union of subintervals.  We will not worry about the end points. We can adopt any  convention that makes it a true partition. For each partition  $\cP$ of the unit interval into $m$  subintervals $\{J_{i}\}$
 there is a  corresponding  partition of $D$ into  $m^2$ sub-squares $\{J_i\times J_j\}$. 
 
 \medskip
 For each integer $m$ we have the  special partition $\cP_m$ of the unit interval into $m$ equal subintervals and they will be denoted by $J^m_{i}=[\frac{i-1}{m},\frac{i}{m}]$.  We denote by ${\cK}_m\subset \cK$ the space of symmetric kernels of the form
 $$
 f(x,y)=\sum_{i,j} f_{i,j}{\bf 1}_{J^m_{i}}(x){\bf 1}_{J^m_{j}}(y)
 $$
 This provides a faithful representation of  symmetric matrices of size $m\times m$ as elements of $\cK_m$. If $|f_{i,j}|\le \ell$, then
 the corresponding $f\in {\cK}^\ell_m=\cK^\ell\cap\cK_m$.
 The following is a simple, but useful lemma.
 
\begin{lemma}\label{partition}
 If   $f(x,y)=\sum f_{i,j}{\bf 1}_{J_i}(x){\bf 1}_{J_j}(y)$ and  $g(x,y)=\sum g_{i,j}{\bf 1}_{J_i}(x){\bf 1}_{J_j}(y)$  are both simple functions with respect to the same partition $J_1,\ldots, J_m$ of $[0,1]$, then in calculating the distance
$$
d_\square(f,g)=\sup_{A\times B}\biggl|\int_{A\times B} (f(x,y)-g(x,y))dxdy\biggr|
$$
the sets $A,B$ can be restricted to sets of the form $\bigcup_{i\in \cN}J_i$ where $\cN\subset \{1,2,\ldots,m\}$.
\end{lemma}
\begin{proof}
First note that\begin{align*}
\int_{A\times B} (f(x,y)-g(x,y))dxdy&=\sum_{i,j}\int_{A\cap J_i\times B\cap J_j} (f(x,y)-g(x,y))dxdy\\
&=\sum_{i,j}|A\cap J_i| |B\cap J_j| (f_{i,j}-g_{i,j})\\
&=\sum_{i,j}a_ib_j |J_i| | J_j|(f_{i,j}-g_{i,j})
\end{align*}
where $a_i=\frac{|A\cap J_i|}{|J_i|}$ and $b_i=\frac{|B\cap J_i|}{|J_i|}$. It is now clear that the supremum of the absolute value is achieved when each $a_i$  and $b_j$ is either $0$ or $1$.
\end{proof}
\begin{remark}
If $f$ and $g$ are defined in terms of two different partitions of  $[0,1]$,  they can both be viewed as defined with respect to the finite partition which is their common refinement.
\end{remark}

\medskip
Since our large deviation result on $\cS$ is deduced from a large deviation result on $\cK$, we need some continuity property of the map
$k\rightarrow S(k)$ of $\cK\to\cS$.

\begin{lemma}\label{cont}
For any $\ell<\infty$,  the map $k\to S(k)$ from $\cK\to \cS$ is  continuous in the cut topology when restricted to  ${\cK}^\ell=\{k: |k(x,y)|\le \ell\}$.
\end{lemma}
\begin{proof}
If $f(\lambda)=0$ near $0$ then $\lambda^{-3}f(\lambda)$ can be approximated uniformly on $[-\ell,\ell]$ by a polynomial in $\lambda$. Therefore $f(\lambda)$ is approximated by a polynomial involving only powers $\lambda^m$ for $m\ge 3$.   Since we have a bound on $\sum_{\lambda\in S}|\lambda|^2$ it is enough to show that 
$$
k\to \sum_{\lambda \in S(k)}\lambda^m
$$
are continuous maps of $\cK^\ell\to \cS$ for each $m\ge 3$. It is elementary to check that since $|k(x,y)|\le \ell $, the maps
$$
k\to \int_{[0,1]^m} k(x_1,x_2)\cdots k(x_{m-1},x_m)k(x_m,x_1) dx_1\ldots dx_m=\sum_{\lambda\in S(k)}\lambda^m
$$
are continuous maps in the cut topology from $\cK^\ell\to\cS$ provided $m\ge 3$. We start with
$$
 \int_{[0,1]^m} k_n(x_1,x_2)\cdots k_n(x_{m-1},x_m)k_n(x_m,x_1) dx_1\ldots dx_m
 $$
 which can be written as 
 $$
 \int_{[0,1]^{m-2}}\int dx_3\ldots dx_m\int_{[0,1]^2} k_n(x_1,x_2)\phi_n(x_1,x_3,\ldots,x_n)\psi_n(x_2,x_3,\ldots,x_n)dx_1dx_2
$$
where $\phi_n$ and $\psi_n$ are uniformly bounded. If $d_\square(k_n,k)\to 0$, we  can then replace $k_n(x_1,x_2)$ by $k(x_1,x_2)$. This is repeated for each factor.
\end{proof}
   
We will also need the following  lemmas: a multicolor version of Szemer\'{e}di's regularity lemma for graphs that can be found in \cite{ks96} and  its consequence.  

\begin{lemma}\label{KS}
Given any $\epsilon >0$ and  integers $r$ and  $m$, there exists $M$  and $n_0$ such that if the edges of a graph ${\cG}_n$ of size $n\ge n_0$  are colored with any one of $r$ colors, then the vertex set can be partitioned into sets  $V_0, \ldots, V_p$ for some $p$ in the range $m\le p\le M $ so that $|V_0|\le \epsilon n$, $|V_1|=|V_2|=\cdots=|V_p|=K$ and all but at most $\epsilon p^2$ pairs $(V_i,V_j)$ satisfy the following regularity condition. For any $X\subset V_i, Y\subset V_j$ with $i,j\ge 1$ and  $|X|,|Y|\ge \epsilon K$ we have
$$
|d_\nu(X,Y)-d_\nu(V_i,V_j)|<\epsilon
$$
where $d_\nu(X,Y)$ is the proportion of edges between $X$ and $Y$ that are colored with color $\nu$.
\end{lemma}
\begin{lemma}\label{KS2}
For any $\ell<\infty$ and  $\epsilon>0$, there is a compact set of simple functions $W_{\ell, \epsilon} \subset\cK$ and $n_0(\epsilon, \ell)$ such that for $n\ge n_0$ and any $k\in \cK^\ell_n$, there exists $f\in W_{\ell, \epsilon}$ and $\sigma\in\Pi(n)$ such that
$$
d_\square(\sigma k,f)<\epsilon.
$$
\end{lemma}
\begin{proof}
Let $\epsilon>0$ be given. Let the integer $r$ be chosen such that $\frac{\ell}{r}<\frac{\epsilon}{2}$. Let $\epsilon'=\frac{\epsilon}{12r\ell}$ and $m = \frac{1}{\epsilon'}$. Let $x_{i,j}$ be the value of $k$ in the rectangle $J_i^n \times J_j^n$. 
The interval $[-\ell, \ell]$ is divided in to $r$ disjoint equal intervals of length $\frac{2\ell}{r}$ and the edge $(i,j)$ is colored according to the interval into which $x_{i,j}$ falls. The \lq color\rq\, is defined  as the value of the mid point of the interval. The color of 
the edge $(i,j)$ is then a real number  $y_{i,j}$ which can equal  any one from the finite set $z_1, z_2, \ldots, z_r$ and   $|y_{i,j}-x_{i,j}|\le \frac{\epsilon}{2}$.  We apply the multicolor version of Szemer\'{e}di's regularity theorem (Lemma \ref{KS}) with parameters $(\epsilon', m, r)$ to obtain a partition $V_0,V_1, \dots V_p$  of $\{1,2,\ldots, n\}$ with the following properties. $|V_0|=K' \le \epsilon'\,n$. $|V_1|=|V_2|=\cdots=|V_p|=K$. If  $a_{s,i,j}$ be the proportion of edges between  $V_i$ and $V_j$ that have color $z_s$, then for all but $\epsilon' p^2$ pairs $(V_i,V_j)$, $i, j\ge 1$ $i\ne j$, for any two  subsets $X\subset V_i$, $Y\subset V_j$ with $|X|\ge \epsilon' K, |Y|\ge \epsilon' K$ the proportion $b_{s,i,j}$  of edges of color $z_s$,  between $X$ and $Y$ satisfies 
$$
\sup_s |b_{s,i,j}-a_{s,i,j}|\le \epsilon '
$$
Let us divide the unit interval into subintervals $J_0,\ldots, J_p$ where $J_0=[0,\frac{K'}{n}]$ and for $i\ge 1$, $J_i=[\frac{K'+(i-1)K}{n}, \frac{K'+iK}{n}]$. We construct a function $f\in \cK^\ell$ as
$$
f(x,y)=\sum_{s=1}^r  z_s f(s,x,y) 
$$
where
$$
f(s,x,y)= \sum_{i,j=0}^p    a_{s,i,j} {\bf 1}_{J_i}(x) {\bf 1}_{J_j}(y).
$$
For fixed $\ell$ and $\epsilon$ as long as $p$ remains bounded such functions vary over a  compact subset of $L^1(D)$.

\medskip
For the permutation $\sigma\in\Pi(n)$ of the vertices, we rearrange the order of the vertices, so that those  in $V_0$ corresponds to the first $K'$ indices and for $1\le i\le p$, those in $V_i$  correspond respectively to  indices in the range   $(K'+(i-1)K+1, K'+iK)$. We will denote by $k'$ the image of $\{x_{\sigma(i),\sigma(j)}\}$ in $\cK^\ell_n$. We define $c(s, i,j)=1$ if $x_{\sigma(i),\sigma(j)}$ belongs to the interval with mid point $z_s$. Otherwise it is $0$. With
$$
k'(s,x,y)=\sum_{i,j=1}^n c(s,i,j) J^n_i(x)J^n_j(y)
$$
and
$$
k'(x,y)=\sum_{i,j=1}^n x_{\sigma(i),\sigma(j)} J^n_i(x)J^n_j(y)
$$
we have that for each $x,y$, 
$$
\biggl|k'(x,y)-\sum_s z_s k'(s,x,y)\biggr|\le \frac{\epsilon}{2}.
$$
We need to estimate
\begin{align*}
\biggl|\int_{A\times B}  [k'(x,y)-f(x,y)]dxdy\biggr|\le \frac{\epsilon}{2}+ r\ell \sup_{1\le s\le r}\biggl|\int_{A\times B}  [k'(s,x,y)-f(s,x,y)]dxdy\biggr|.
\end{align*}
For each value of $s$, we will estimate
$$
\int_{A\times B} [ k'(s,x,y)-f(s,x,y)]dxdy=\sum_{i,j=0}^p \int_{A\cap J_i\times B\cap J_j}  [k'(s,x,y)-f(s,x,y)]dxdy.
$$
The summation over $(i,j)$ will be split into several groups. $F_0 = \{(i,i): i\ge 1\}$. $F_1=\{(i,j): i=0\}\cup \{(i,j): j=0\}$. $F_2=\{(i,j)\}$  is  the collection of   at most $\epsilon' p^2$ exceptional pairs  from Lemma \ref{KS}. $F_3=\{(i,j)\}$ for which either $|A\cap J_i|\le \epsilon' |J_i|$ or $|B\cap J_j|\le \epsilon' |J_j|$. $F_4$ will be the rest. Then: 
$$
\biggl|\sum_{(i,j)\in F_0} \int_{A\cap J_i\times B\cap J_j}  [k'(s,x,y)-f(s,x,y)]dxdy\biggr|\le \frac{pK^2}{n^2}\le \frac{1}{p} \le \epsilon'.
$$
$$
\biggl|\sum_{(i,j)\in F_1} \int_{A\cap J_i\times B\cap J_j}  [k'(s,x,y)-f(s,x,y)]dxdy\biggr|\le 2 \epsilon'.
$$
$$
\biggl|\sum_{(i,j)\in F_2} \int_{A\cap J_i\times B\cap J_j}  [k'(s,x,y)-f(s,x,y)]dxdy\biggr|\le \epsilon' p^2 |J_i||J_j|\le \epsilon' p^2\frac{1}{p^2}=\epsilon'.
$$
$$
\biggl|\sum_{(i,j)\in F_3} \int_{A\cap J_i\times B\cap J_j}  [k'(s,x,y)-f(s,x,y)]dxdy\biggr|\le \epsilon'\sum_{i,j}|J_i||J_j|\le\epsilon'.
$$
Finally in the remaining set $F_4$, since  $|A\cap J_i|\ge \epsilon' |J_i|$ and $|B\cap J_j|\ge \epsilon' |J_j|$,  it follows that
$$\biggl|\int_{A\cap J_i\times B\cap J_j}  [k'(s,x,y)-f(s,x,y)]dxdy\biggr|\le \epsilon' |J_i||J_j|$$
 and hence
$$
\biggl|\sum_{(i,j)\in F_4} \int_{A\cap J_i\times B\cap J_j}  [k'(s,x,y)-f(s,x,y)]dxdy\biggr|\le \epsilon'\sum_{i,j}|J_i||J_j|\le\epsilon'.
$$
Adding them up gives
$$
\biggl|\int_{A\times B}  [k'(x,y)-f(x,y)]dxdy\biggr|\le \frac{\epsilon}{2}+ 6 r\ell \epsilon' \le \epsilon
$$
\end{proof}

\section{Lowerbound} 
The lower bound for ${\tilde Q}_n$ on $\tilde \cK$ can be proved by proving a lower bound for $Q_n$ on $\cK$ and it can be done  without truncation and under the (weaker) assumption  that $\int e^{\theta \, x^2} d\mu(x)<\infty$ for some $\theta$. This implies a lower bound $h(x)\ge c|x|^2$  when $|x|$ is large.
\begin{theorem} Under the weaker assumption that for some $\theta>0$,
$$
\int e^{\theta \, x^2} d\mu(x)<\infty,
$$
it follows that for any $f\in \cK$ such that 
$$
I(f)=\frac{1}{2}\int h(f(x,y)) dxdy<\infty
$$
and for any $\delta>0$,
$$
\liminf_{n\to\infty} \frac{1}{n^2}\log Q_n[ d_\square(k, f)<\delta]\ge - I(f)
$$
\end{theorem}
\begin{proof}
Since $h$ is a convex function of its argument, for any integer $q$ we can replace  $f$  by a  simple function
 $g$  of its averages over $J^q_r\times J^q_s$ so that 
$$
g(x,y)=\sum_{r,s=1}^q g_{r,s}{\bf 1}_{J^q_r}(x){\bf 1}_{J^q_s}(y)
$$
where
$$
g_{r,s}=q^2\int_{J^q_r\times J^q_s} f(x,y) dxdy.
$$
For  any $q$, $I(g)\le I(f)$ and for large $q$,  $d_\square(f,g)<\frac{\delta}{2}$. It suffices to show that
$$
\liminf_{n\to\infty}\frac{1}{n^2}\log Q_n[d_\square(g,k)\le \frac{\delta}{2}]\ge -I(g).
$$
We have
$$
k(x,y)=\sum_{i,j=1}^n x_{i,j}{\bf 1}_{J^n_i}(x){\bf 1}_{J^n_j}(y)
$$
and we can view both $k$ and $g$ as members of $\cK_{nq}$.   The distance $d_\square (g,k)$ can be computed as
$$
d_\square (g,k)=\sup_{A, B}\biggl|\int_{A\times B} (g(x,y)-k(x,y))dxdy\biggr|
$$
where $A$ and $B$ are taken to be unions of sub-collections of intervals of the form $\{[\frac{i-1}{nq},\frac{i}{nq}]\}$. There are exactly $2^{nq}\times 2^{nq}$ such pairs $A,B$. We now tilt the measure so that $\{x_{i,j}\}$ remains  symmetric and $\{x_{i,j}\}$   for   $j\ge i$ are   still independent but the distribution of $x_{i,j}$ is tilted  from $\mu$ to $\mu_{i,j}$ given by
$$
\mu_{i,j}(dx)=\frac{1}{M(\theta_{r,s})}\exp[\theta_{r,s}\, x]\mu(dx)
$$ 
for $(r-1)n< iq \le rn$ and $(s-1)n< j q\le sn$ where $\theta_{r,s}= h'(g_{r,s})$, or equivalently $g_{r,s} = \frac{M'(\theta_{r,s})}{M(\theta_{r,s})}$ for 
$1\le r,s\le q$. Let $Q_n^g$ be the law of the new $k$. 
The law of large numbers applies to each such pair $A,B$ with uniform (in $A$ and $B$) exponential error bounds of  $e^{-cn^2+o(n^2)}$ for some $c>0$. Therefore,  
$$
Q^{g}_n\biggl[\biggl|\int_{A\times B} (g(x,y)-k(x,y))dxdy\biggr|\ge \frac{\delta}{2}\biggr]\le e^{-c(\delta) n^2+o(n^2)}
$$
Since $2^{2nq}\ll e^{cn^2}$, it follows that
$$
Q^{g}_n[d_\square(g,k)\ge \frac{\delta}{2}]\to 0
$$

\medskip
The relative entropy of $Q^{g}_n$ with respect to $Q_n$ is easily computed to be  $n^2I(g)$.

\medskip
The following entropy lower bound using Jensen's inequality, establishes the large deviation lower bound. Suppose $\alpha$, $\beta$ are two probability measures and $\beta \ll \alpha $ with $H=\int \phi\log\phi \,d\alpha<\infty$ where $\phi=\frac{d\beta}{d\alpha}$. Since $y\log y \ge -e^{-1}$ for all $y\ge 0$, $\int |\phi\log\phi| d\alpha\le H+2e^{-1}$. Therefore
\begin{align*}
\alpha(A)&\ge  \int_A \phi^{-1} d\beta=\int_A \exp[-\log \phi]d\beta\\
&=\beta(A)\frac{1}{\beta(A)}\int_A \exp[-\log \phi ]d\beta\\
&\ge \beta(A)\exp\biggl[-\frac{1}{\beta(A)}\int_A \phi \log \phi d\alpha\biggr]\\
&\ge \beta(A)\exp\biggl[-\frac{1}{\beta(A)}\int |\phi\log \phi|d\alpha\biggr]\\
&\ge \beta(A)\exp\biggl[-\frac{1}{\beta(A)}[H+2e^{-1}]\biggr]
\end{align*}
Taking $\alpha=Q_n$ and $\beta=Q^g_n$ and $A=\{k: d_\square(k,g)\le \frac{\delta}{2}\}$ we have $H=n^2I(g)$ and $\beta(A)\simeq 1$.
\end{proof}

\section{Upperbound}
We assume that $|x_{i,j}|\le \ell$. According to a result of Lov\'{a}sz and Szegedy \cite{lovaszszegedy07},  for any finite $\ell$,  the set $\tilde\cK^\ell =\{{\tilde f}: \sup_{x,y}|f(x,y)|\le \ell\}$ is compact in $\tilde K$.
Therefore in order to prove the large deviation upper bound for ${\tilde Q}_n$ with rate function $I(f)$  it is sufficient to prove the local version of the upper bound.

\begin{theorem}
Let $\tilde f\in{\tilde \cK}$. Then
$$
\limsup_{\epsilon\to 0}\limsup_{n\to\infty}\frac{1}{n^2}\log {\tilde Q}_n[{\tilde k}: d_\square({\tilde k}, {\tilde f})<\epsilon]\le -I({\tilde f})
$$
\end{theorem}
\begin{proof}
The theorem  is proved in several steps.  Closed balls $B(f,\epsilon)$ in $\cK$ of the form $\{k: d_\square(k,f)\le \epsilon\}$ are weakly closed. It is not hard to prove that
\begin{equation}\label{wub}
\limsup_{\epsilon\to 0}\limsup_{n\to\infty}\frac{1}{n^2}\log Q_n[ k: d_\square( k, f)\le \epsilon]\le -I(f).
\end{equation}
The argument goes as follows. Assume $I(f)<\infty$. Given $\delta>0$ pick a nice $g$ such that
$$
\frac{1}{2}\bigg[\int f(x,y) g(x,y)dxdy-\int  \log \biggl[\int e^{z\,g(x,y)}\mu(dz)\biggr] dx dy\bigg] \ge I(f)-\delta
$$
and apply Cram\'{e}r type estimate using the moment generating function for the half space
$$
H_{f,g,\epsilon}=\bigg\{ k: \langle k, g\rangle \ge \inf_{k'\in B(f,\epsilon)}\langle k', g\rangle\bigg\}
$$
which contains $B(f,\epsilon)$. This gives 
$$
Q_n[B(f,\epsilon)]\le  Q_n[H_{f,g,\epsilon}]\le \exp\biggl[-\frac{n^2}{2} \inf_{k\in B(f,\epsilon)}\langle k, g\rangle\biggr] \bbE^{Q_n}\exp\biggl[\frac{n^2}{2}\langle k, g\rangle\biggr]
$$
It is easy to see (because $x_{i,j}=x_{j,i}$), that
\begin{align*}
\limsup_{n\to\infty}\frac{1}{n^2}\log \bbE^{Q_n}\exp\biggl[\frac{n^2}{2}\langle k, g\rangle\biggr]=\frac{1}{2}\int_D \log \biggl[\int \exp[zg(x,y)]\mu(dz)\biggr]dxdy
\end{align*}
and
$$
\lim_{\epsilon\to 0}\inf_{k\in B(f,\epsilon)}\langle k, g\rangle=\langle f, g\rangle
$$
We can let $\delta\to 0$ at the end. If $I(f)=\infty$, pick $g$ such that
$$
\frac{1}{2}\bigg[\int f(x,y) g(x,y)dxdy-\int  \log \biggl[\int e^{z\,g(x,y)}\mu(dz)\biggr]dxdy\bigg] \ge L
$$
and let $L\to\infty$ in the end.

\medskip
   This proves \eqref{wub}. But to prove it for balls in $\tilde\cK$ we need to estimate the probability of the pre-image in $\cK$ of a ball in $\tilde \cK$.
This amounts to showing  that
$$
\limsup_{\epsilon\to 0}\limsup_{n\to\infty}\frac{1}{n^2}\log  Q_n[k: \inf_{\sigma \in G}d_\square(k ,\sigma f)<\epsilon]\le -I({\tilde f})
$$ 
We start with the set
$$
U(\tilde f, \epsilon)=\cup_{\sigma\in G} B(\sigma f, \epsilon).
$$
We saw that for each  $f\in \cK$,  (\ref{wub}) holds. We need to replace the union over $\sigma\in G$ by a union over a finite collection. According to Lemma \ref{KS2}  there is a compact set
$W_{\ell,\epsilon}\subset \cK$ such that its orbit by $\Pi(n)$ nearly covers $\cK_n^\ell$ for all sufficiently large $n$. More precisely for $n\ge n_0(\epsilon,\ell)$, 
$$
\sup_{k\in \cK_n^\ell } \inf_{\sigma\in\Pi(n)} \inf_{g\in W_{\ell,\epsilon}}d_\square(\sigma g, k)<\epsilon.
$$
$W_{\ell,\epsilon}$ is compact in the cut topology (and even in $L^1$) and can be covered by the union of balls $B(g,\epsilon)$ of radius $\epsilon$ centered around $g$ from  a finite collection $ F_{\ell, \epsilon}$. Therefore for sufficiently  large $n$
$$
\cup_{\sigma\in \Pi(n)}\cup_{g\in F_{\ell,\epsilon}} B(\sigma g, 2\epsilon)\supset \cK_n^\ell
$$
We need to estimate the probability under $Q_n$ of 
$$
\cup_{\sigma'\in \Pi(n)}\cup_{g\in F_{\ell, \epsilon}}\cup_{\sigma\in G} [B(\sigma' g, 2\epsilon)\cap B(\sigma f,\epsilon)]
$$
which is at most $n!$ times 
$$
\sup_{\sigma'\in \Pi(n)}Q_n[  \cup_{g\in F_{\ell, \epsilon}}\cup_{\sigma\in G}  [B( \sigma' g, 2\epsilon)\cap B(\sigma f, \epsilon)]]
$$ 
Since $F_{\ell,\epsilon}$ is  a finite set   independent of $n$ and $Q_n$ is invariant under $\Pi(n)\subset G$, it is enough to show that for each $g\in F_{\ell,\epsilon}$
$$
\limsup_{\epsilon\to 0}\limsup_{n\to\infty}\frac{1}{n^2}\log Q_n[ B(g, 2\epsilon)\cap \cup_{\sigma\in G}  B(\sigma f, \epsilon)]\le -I(f)
$$
If the intersection is nonempty, then there is a $\sigma'$ such that $B(g, 2\epsilon)\subset B(\sigma' f, 5\epsilon) $.
Thus, from  (\ref{wub}), the invariance of $I(\sigma f)$ under $\sigma\in G$ and the lower semicontinuity of $I(\cdot)$ on $\tilde \cK$ we see that there is a $\sigma'\in G$ such that 
\begin{align*}
\limsup_{\epsilon\to 0}&\limsup_{n\to\infty}\frac{1}{n^2}\log Q_n[ B(g, 2\epsilon)\cap \cup_{\sigma\in G}  B(\sigma f, \epsilon)]\\
&\le \limsup_{\epsilon\to 0}\limsup_{n\to\infty}\frac{1}{n^2}\log Q_n[ B(\sigma' f, 5\epsilon)]\\
&=-I(\sigma' f) = -I(f).
\end{align*}
\end{proof}

\section{Truncation}
Given $X=\{x_{i,j}\}$ we truncate it at level $\ell$. Let  $x=f_\ell(x)+g_\ell(x)$ where
\begin{align*}
f_\ell(x)=
\begin{cases}
x &\ {\rm if}\ |x|\le \ell\\
\ell &\ {\rm if}\ x\ge \ell\\
-\ell & \  {\rm if }\ x\le -\ell
\end{cases}
\end{align*}
and $g_\ell(x)= x-f_\ell(x)$.  There is a corresponding  decomposition  of $k=f_\ell(k)+ k-f_\ell(k)$ for $k\in \cK$.  We have the estimate 
$$
\limsup_{\ell\to\infty} \limsup_{n\to\infty} \frac{1}{n^2}\log \bbE \exp\Big[\theta\sum_{i,j}[x_{i,j}-f_\ell(x_{i,j})]^2\Big]=0 
$$
for any $\theta>0$. If $\{f_\ell(x_{i,j})\}$  and $\{x_{i,j}\}$ are  mapped respectively into $k_\ell$ and $k$,  we have super exponential estimates on   $\Delta(\ell)=\int_D |f_\ell(k(x,y))-k(x,y)|^2 dxdy$:
$$
Q_n[ \Delta (\ell) \ge \epsilon]\le e^{-n^2 \theta\epsilon^2}\bbE\exp\Big[\theta \sum_{i,j}  [x_{i,j}-f_\ell(x_{i,j})]^2 \Big].
$$
Therefore
$$
\limsup_{\ell\to\infty}\limsup_{n\to\infty} \frac{1}{n^2}\log Q_n[ \Delta(\ell)\ge \epsilon]\le - \theta\epsilon^2.
$$
Since $\theta>0$ is arbitrary, for any $\epsilon>0$,
$$
\limsup_{\ell\to\infty}\limsup_{n\to\infty} \frac{1}{n^2}\log Q_n[ \Delta(\ell)\ge \epsilon]=-\infty.
$$
It is easy to complete the proof of Theorem \ref{main2} using the above identity and the lower-semicontinuity of $I$.

The difference between the eigenvalues of two kernels $k_1$ and $k_2$ can be easily controlled by the Hilbert-Schmidt norm
of the difference $k_1-k_2$. In fact using the variational formula for successive eigenvalues of a compact self adjoint
operator, 
$$
\lambda_{j+1}(A)=\inf_{\cM:\atop codim(\cM)=j}\sup_{\|z\|=1\atop z\in\cM} \langle Az,z\rangle
$$
it is easily  seen  that if $\lambda_j(A)$ and $\lambda_j(B)$ are the $j$th largest positive eigenvalues of  two compact self adjoint operators $A$ and $B$
$$
|\lambda_j(A)-\lambda_j(B)|\le \|A-B\|\le \|A-B\|_{HS}
$$  
and a similar formula holds for negative eigenvalues as well. Clearly if $k_1,k_2\in\cK$, then $d(S(k_1),S(k_2))\le \|k_1-k_2\|\le \|k_1-k_2\|_{HS}$.

\medskip
If  $C\in \cS$ is a closed set, and we truncate $X$ at level $\ell$  and denote by $k_\ell$ and $k$ the two images in $\cK$, then
$$
P_n(C)=Q_n[S(k)\in C]\le Q_n [S(k_\ell)\in C^\epsilon]+Q_n[d(S(k_\ell),S(k))\ge \epsilon]
$$
Applying the upper bound for the truncated version,  
$$\limsup_{n\to\infty}\frac{1}{n^2}\log P_n[C]\le \max\{-\inf_{S\in \overline{C^\epsilon}} J_\ell(S), -c(\ell,\epsilon)\}$$
where
$$
c(\ell,\epsilon)=-\limsup_{n\to\infty} \frac{1}{n^2}\log Q_n[d(S(k_\ell),S(k))\ge \epsilon]
$$
and for any $\epsilon>0$,
$$
\lim_{\ell\to\infty}c(\ell,\epsilon)=-\infty.
$$
Moreover
$$
J_\ell(S)=\inf_{k: S(k)=S}I_\ell(k); \quad I_\ell(k)=\int_D h_\ell(k(x,y))dxdy
$$
and
$$
h_\ell(z)=\sup[z\theta-\log\int \exp[\theta f_\ell(x)]\mu(dx)].
$$
The upper bound is now easily established by letting $\ell\to\infty$ and then letting $\epsilon\to 0$. We observe that the superexponential estimate implies $c(\ell,\epsilon)\to \infty$ as $\ell\to\infty$ for any $\epsilon >0$. This completes the proof of Theorem \ref{main}.

\noindent 

\end{document}